\documentclass[psamsfonts,12pt]{amsart}
\usepackage{graphicx}
\usepackage{amssymb,amsfonts}
\usepackage{amsmath, amscd}
\usepackage[all]{xy} 
\usepackage{enumerate}
\usepackage{mathrsfs}
\usepackage{amsthm}
\usepackage{eufrak}
\usepackage{fullpage} 
\usepackage{hyperref}

\newtheorem{thm}{Theorem}[section]
\newtheorem*{coroa}{Corollary A}
\newtheorem*{corob}{Corollary B}
\newtheorem*{coroc}{Corollary C}
\newtheorem*{thmd}{Theorem D}

\newtheorem{lem}[thm]{Lemma}

\newtheorem{quest}[thm]{Question}

\theoremstyle{definition}

\theoremstyle{remark}

\numberwithin{equation}{section}

\title{On the normality of secant varieties}

\author{Brooke Ullery}

\address{Department of Mathematics\\ University of Michigan\\ 530 Church Street,
Ann Arbor, MI  48109-1043}
\email{bullery@umich.edu}

\newcommand{\Sym}{{{\textrm{Sym}}}}
\newcommand{\Supp}{{{\textrm{Supp}}}}

\newcommand{\A}{{{{\mathcal{A}}}}}
\newcommand{\B}{{{{\mathcal{B}}}}}
\newcommand{\I}{\mathcal{I}}
\renewcommand{\P}{{{\mathbb{P}}}}
\newcommand{\E}{{{\mathcal{E}_\L}}}
\renewcommand{\L}{{{\mathcal{L}}}}

\newcommand{\iv}{{{^{-1}}}}
\newcommand{\newword}[1]{\textbf{\emph{#1}}}

\newcommand{\str}{\mathcal{O}}

\newcommand{\sct}{\Sigma(X,\L)}
\renewcommand{\H}{X^{[2]}}
\newcommand{\surj}{\twoheadrightarrow}

\newcommand{\U}{{N^*_{F_x/\P(\E)}}}
\newcommand{\Ud}{{N_{F_x/\P(\E)}}}
\newcommand{\T}{{T_{\P(\E)/\H}\big|_{F_x}}}
\newcommand{\PE}{{\P(\E)}}
\newcommand{\arinj}{\ar@{^{(}->}}
\newcommand{\arsurj}{\ar@{->>}}

\begin{document}

\maketitle 

\section*{Introduction}
The purpose of this paper is to show that the secant variety to a projective variety embedded by a sufficiently positive line bundle is a normal variety.  In particular, this confirms the vision and completes the results of Vermeire in \cite{V} and renders unconditional the results in \cite{SV2}, \cite{SV}, \cite{V5}, and \cite{V4}.\footnote{This question of the normality of the secant variety came up in 2001 when a proof was proposed by Vermeire \cite{V}.  However, in 2011, Adam Ginensky and Mohan Kumar pointed out that the proof was erroneous, as explained in Remark 4 of \cite{SV}.}

Let $$X \subset \P(H^0(X,\L)) = \P^r$$  be a smooth variety over an algebraically closed field of characteristic zero, embedded by the complete linear system corresponding to a very ample line bundle $\L$.  We define the secant variety $$\Sigma(X, \L) \subset \P^n$$  to be the Zariski closure of the union of 2-secant lines to X in $\P^r$.  As secant varieties are classical constructions in algebraic geometry, there has been a great deal of work done in an attempt to understand their geometry.  Recently, there has been interest in determining defining equations and syzygies of secant varieties \cite{CGG} \cite{CGG2} \cite{CS} \cite{SV2} \cite{SV} \cite{V2}, motivated in part by questions in algebraic statistics \cite{GSS} \cite{SS} and algebraic complexity theory \cite{L} \cite{LW}.  In this paper, we focus on the singularities of secant varieties, using the comprehensive geometric description developed by Bertram \cite{B} and Vermeire \cite{V}.

If the embedding line bundle $\L$ is not sufficiently positive, the behavior of the singularities of $\Sigma(X, \L)$ can be quite complicated.  For example, the secant variety is generally singular along $X$, but if four points of $X$ lie on a plane, then three pairs of secant lines will intersect away from $X$. In some cases this will create additional singularities at those intersection points on $\sct$.  In more degenerate cases, the secant variety may simply fill the whole projective space, e.g. the secant variety to any non-linear plane curve. However, if $\L$ is sufficiently positive, we will see that $\sct$ will be singular precisely along $X$.  As $\L$ becomes increasingly positive, it is natural to predict that the singularities of $\sct$ will become easier to control.

We start by stating some concrete special cases of the main theorem.  In the case of curves, normality of the secant variety only depends on a degree condition:

\begin{coroa}
Let $X$ be a smooth projective curve of genus $g$ and $\L$ a line bundle on $X$ of degree $d$.  If $d\geq 2g+3$, then $\sct$ is a normal variety.
\end{coroa}

Moreover, in the example of canonical curves, we have a stronger result not covered by the above proposition:

\begin{corob}
Let $X$ be a smooth projective curve with Clifford index $\emph{Cliff}(X)\geq 3$.  Then $\Sigma(X,\omega_X)$ is a normal variety.
\end{corob}

In particular, the above implies that the secant variety to a general canonical curve of genus at least 7 is normal.

More generally, we can also give a positivity condition on embeddings of higher dimensional varieties to ensure that the secant variety is normal:

\begin{coroc}
Let $X$ be a smooth projective variety of dimension $n$. Let $\A$ and $\B$ be very ample and nef, respectively, and $$\L=\omega_X \otimes  \A^{\otimes 2(n+1)} \otimes \B.$$ Then $\sct$ is a normal variety.
\end{coroc}

Before we state the main theorem, we must define $k$-very ampleness, a rough measure of the positivity of a line bundle:  

A line bundle $\L$ on $X$ is \newword{$k$-very ample} if every length $k+1$ 0-dimensional subscheme $\xi \subseteq X$  imposes independent conditions on $\L$, i.e. $$H^0(\L) \to H^0(\L \otimes \str_{\xi})$$ is surjective.\footnote{Some sources, e.g. \cite{SV}, \cite{SV2}, \cite{V}, and \cite{V2}, call this property $(k+1)$-very ampleness.}  In other words, $\L$ is 1-very ample if and only if it is very ample, and for any positive $k$, $\L$ is $k$-very ample if and only if no length $k+1$ 0-dimensional subscheme of $X$ lies on a $(k-1)$-plane in $\P(H^0(\L))$.

Our main result is the following:

\begin{thmd}
Let $X$ be a smooth projective variety, and $\L$ be a 3-very ample line bundle on $X$.  Let $m_x$ be the ideal sheaf of $x\in X$.  Suppose that for all $x\in X$ and $i>0$, the natural map $$\emph{\Sym}^i H^0 (\L \otimes {m}_x^{2}) \to H^0 (\L^{\otimes i} \otimes {m}_x^{2i})$$ is surjective.\footnote{Note that this map is surjective for every $i$ if and only if $b_x^*\L \otimes \str(-2E_x)$ (or simply $\L(-2x)$ when $X$ is a curve) is normally generated, where $b_x$ is the blow-up map of $X$ at $x$, and $E_x$ is the corresponding exceptional divisor.}
Then $\sct$ is a normal variety.
\end{thmd}

Though the hypothesis is a bit abstract, we will show how the more accessible corollaries follow from the theorem in section 3.

I would like to thank my advisor, Rob Lazarsfeld, for suggesting the problem and for the hours of discussions along the way.  I would also like to thank my other advisor, Karen Smith, for her encouragement and many helpful discussions.  I am also grateful to Lawrence Ein, Ian Shipman, Jessica Sidman, and Pete Vermeire for their suggestions and fruitful conversations.

The work for this paper was partially supported by NSF RTG grant DMS 0943832

\section{The geometry of the secant variety}

First we will describe the geometric setup, detailed in the case of curves in \cite{B}, and extended to higher dimensions in \cite{V}.  Our notation will be the same as in the introduction.

Let $\L$ be a very ample line bundle on a smooth variety $X$ and $$X \hookrightarrow \P(H^0(\L)) = \P^r$$ the corresponding embedding, thinking of the points of $\P(H^0(\L))$ as the one dimensional quotients of $H^0(\L)$.  Let $X^{[2]}$ denote the Hilbert scheme of length 2 zero-dimensional subschemes of $X$.  Recall that $\H$ is smooth, and its universal subscheme is the incidence variety 
\begin{equation}\label{phi} \Phi  = \{(x,\xi) \in X \times \H : x \in \xi\} \cong bl_{\Delta}(X \times X),
\end{equation} the blowup of $X \times X$ along the diagonal.  Moreover, we have the Cartesian square 
\begin{equation} \xymatrix@=10pt{bl_{\Delta}(X \times X) \ar[r] \ar[d] & X\times X \ar[d] \\ \H \ar[r] & \Sym^2 X}\end{equation}
where the vertical arrows are quotients by the involution, and the horizontal maps are the natural ones.  Note that when $X$ is a curve, the horizontal maps are isomorphisms.

Let $q$ and $\sigma$ be the two projections as shown below:
$$\xymatrix@=10pt{& \Phi \ar[dl]_q \ar[dr]^{\sigma} &  
\\ X && \H }.$$
Define the vector bundle $$\E=\sigma_* q^* \L.$$  This rank two vector bundle is tautological in the sense that the fiber of $\E$ over $\xi \in \H$ consists of the global sections of $\L$ restricted to the corresponding subscheme of $X$.  That is, $$\textrm{fiber of } \E \textrm{ over } \xi = H^0(X, \L \otimes \str_\xi).$$
Thus, since $\L$ is very ample, the map $$H^0(\L) \otimes \str_{\H} \to \E$$ is surjective and induces a morphism $$f: \P(\E) \to \P^r.$$  We can think of the points of $\P(\E)$ as pairs $(\xi, H^0(\L\otimes \str_{\xi}) \surj Q)$, where $Q$ is a one-dimensional quotient, and $\xi$ is a point of $\H$.  Thus,
\begin{equation}\label{fdef}f(\xi, H^0(\L\otimes \str_{\xi}) \surj Q) = (H^0(\L) \surj Q) \in \P^r.\end{equation} Notice that the image of $f$ is $\sct$, since the surjections in the image are precisely those which factor through $H^0(\L\otimes \str_{\xi})$ for some $\xi \in \H$.  Let $$t: \P(\E) \surj \sct$$ be $f$ with its target restricted.

The following lemma is adapted from \cite{B} in the case of curves and \cite{V} for higher dimensions.
\begin{lem}\label{iso}
Suppose $\L$ is 3-very ample. Then $t: \P(\E) \to \sct$ is an isomorphism away from $t \iv(X)$.  In particular, $t$ is a resolution of singularities.
\end{lem}

\begin{proof}
For clarity, we first show that $t$ is a bijection away from $t \iv(X)$, which follows nearly immediately from the 3-very ampleness of $\L$:

  Given a length two 0-dimensional subscheme $\xi$, points of the form $(\xi, H^0(\L\otimes \str_{\xi}) \surj Q) \in \P(\E)$ map to the secant line to $\xi$.  Since $\L$ is 3-very ample, no two distinct length two subschemes will correspond to the same secant line.  Thus, the only way for $t$ not to be a bijection away from $t \iv(X)$ would be for two secant lines of $X$ to intersect away from $X$.  This would cause four points of $X$ to lie on a plane in $\P^r$, which contradicts the 3-very ampleness of $\L$.  
  
In order to show that $t$ is actually an isomorphism away from $t \iv(X)$, we need to check that it is an immersion.  This follows in the curve case from Lemma 1.4 of \cite{B}, and in the higher dimensional case from Theorem 3.9 of \cite {V}.  In the former, Bertram proves that it is an immersion directly.  In the latter, Vermeire shows that $\P(\E)$ is isomorphic to the blowup of $\sct$ along $X$, which clearly implies what we need.
\end{proof}

For our purposes, it will be useful to also understand $t \iv(X)$.  Looking at \eqref{fdef}, we see that $$t \iv(X) = f \iv(X) = \{(\xi, H^0(\L\otimes \str_{\xi}) \surj H^0(\L\otimes \str_{x}): x\in \xi\},$$ which is set-theoretically equal to $\Phi$ (defined in (\ref{phi})).  In fact, a lemma of Vermeire implies that it is actually an isomorphism:

\begin{lem}[\cite{V}, Lemma 3.8]
The scheme-theoretic inverse image $t\iv (X)$ is isomorphic to $bl_{\Delta}(X \times X)$.
\end{lem}

From now on, we will refer to $t \iv (X)$ as simply $\Phi$. Notice that $t\big|_{\Phi} = q$, and for $x \in X$, the fiber is $$F_x := t\iv (x) = \{\xi : x\in \xi\} \cong bl_x(X),$$ which is simply $X$ when $X$ is a curve.\footnote{All of the arguments for the remainder of the paper go through in the case of curves by replacing $E_x$ with $x$. From now on, this will be assumed.}

Let $$\pi : \P(\E) \to \H$$ be the projection map.  Notice that $\pi \big|_{\Phi} = \sigma$.  Furthermore, $\pi \big|_{F_x}$ is an isomorphism, as $F_x$ is a section over $\pi(F_x)$.  When the context is clear, we will refer to $\pi(F_x)$, the points of $\H$ whose  corresponding subschemes contain $x$, as simply $F_x$.

To summarize, we have the following two commutative diagrams, which we will keep in mind for the remainder of the paper:

\begin{equation}\label{impdiags}
\xymatrix@=20pt{ F_x \arinj[r] \arsurj[d] & \Phi \arinj[r] \arsurj[d]^q & \PE \arsurj[d]^t \arsurj[dr]^f & \\ \{x\} \arinj[r] & X \arinj[r] & \sct \arinj[r] & \P^r}
\textrm{and    \hspace{.5cm} }
\xymatrix@=20pt{ F_x \arinj[r] \ar[d]^{\cong} & \Phi \arinj[r] \arsurj[d]^{\sigma} & \PE \arsurj[dl] \\ F_x \arinj[r] & \H &}.
\end{equation}

Since $t$ is a resolution of singularities, our strategy for showing $\sct$ is normal is to show $t_*\str_{\P(\E)} = \str_{\sct}$ by exploiting the geometry of $\Phi$ and $F_x$.

\section{Proof of the main theorem}

In this section, we give the proof of Theorem D, continuing with the same notation as in the previous sections.  We begin by observing that the normality of the secant variety is controlled by the geometry of the conormal bundle to $F_x$.

\begin{lem}\label{alpha}
Let $\L$ be a 3-very ample line bundle on $X$.  Let $x \in X$, and let $\alpha_{x,k}$ be the natural map
$$\alpha_{x,k} : \emph{\Sym}^k(T_x^*\P^r) \to H^0(\emph{\Sym}^k N^*_{F_x/\P(\E)}).$$  If $\alpha_{x,k}$ is surjective for all $k>0$ and all $x \in X$, then $\sct$ is a normal variety.
\end{lem}

\begin{proof}
We have the following natural maps of sheaves:

$$\xymatrix@=10pt{\str_{\P^r} \ar[d] \ar@{->>}[r]& \str_{\sct} \ar@{_{(}->}[dl]\\ t_*\str_{\P(\E)} & }.$$
As pointed out at the end of the last section, if $t_*\str_{\P(\E)} = \str_{\sct}$, then $\sct$ is normal.  So we need to show $\str_{\sct} \to t_*\str_{\P(\E)}$ is surjective.  Thus, by the above diagram, it suffices to show $\str_{\P^r} \to t_*\str_{\P(\E)}$ is surjective.

The map $\str_{\P^r} \to t_*\str_{\P(\E)}$ is surjective if and only if the completion of the map is surjective at every point $x\in \sct$.  However, we only need to check this for $x\in X$, since $\P(\E)$ is smooth, and $t$ is an isomorphism away from $t\iv (X)$ by Lemma \ref{iso}.

Let $$\I_x = \textrm{the ideal sheaf of } F_x \subseteq \P(\E)$$ and $$m_x = \textrm{the ideal sheaf of } x \in \P^r.$$  Then by the theorem of formal functions \cite{H}, we need to show that the map 
$$\Psi_x : \lim_{\longleftarrow} \left(\str_{\P^r}/m_x^k\right) \to \lim_{\longleftarrow} \left(H^0\left(\str_{\P(\E)}/\I_x^k\right)\right)$$ is surjective for each $x \in X$.

Consider the following diagram:
\begin{equation}\label{diag} \xymatrix@=19pt{ 0 \ar[r] & m_x^k/ m_x^{k+1} \ar[r] \ar[d]^{\alpha_{x,k}} &\str_{\P^r} / m_x^{k+1} \ar[r]^a \ar[d]^{\Psi_{x,k+1}} & \str_{\P^r} / m_x^{k} \ar[r] \ar[d]^{\Psi_{x,k}}& 0& \\ 0 \ar[r] & H^0 \left( \I_x^k / \I_x^{k+1} \right) \ar[r] & H^0 \left( \str_{\P(\E)} / \I_x^{k+1} \right) \ar[r]^b& H^0 \left( \str_{\P(\E)} / \I_x^{k} \right) \ar[r]^c & H^1 \left( \I_x^k / \I_x^{k+1} \right) \ar[r] & \cdots}.\end{equation}
Note that we have canonical isomorphisms 
$$m_x^k/ m_x^{k+1} \cong \Sym^k(T_x^*\P^r)$$ and $$ \I_x^k / \I_x^{k+1} \cong \Sym^k N^*_{F_x/\P(\E)}.$$ 

We claim that it suffices to show all the vertical maps are surjective for all $k$: Assume the vertical maps are surjective.  Then the snake lemma says that $$\ker \Psi_{x, k+1} \to \ker \Psi_{x, k}$$ is surjective for all $k$.  In particular, the inverse system $(\ker \Psi_{x, k})$ satisfies the Mittag-Leffler condition (see II.9 of \cite{H}).  Thus, by Prop II.9.1(b) of \cite{H}, $\Psi_x$ is surjective.  Thus, we are reduced to showing that the vertical arrows are surjections.

We claim that if the left vertical arrow $\alpha_{x,k}$ is surjective for all $k$, then $\Psi_{x,k}$ is surjective for all $k$.  We show this by induction.  

The base case is $k=1$:  Consider the map $$ \Psi_{x,1} : \str_{\P^r} / m_x \to H^0 \left( \str_{\P(\E)} / \I_x \right) = H^0 (\str_{F_x}) .$$
Since $F_x$ is reduced and irreducible, $h^0 (\str_{F_x}) = 1$, and since $\Psi_{x,1}$ is certainly nonzero, it must be surjective.

Now assume $\Psi_{x,k}$ is surjective.  Then, looking back at (\ref{diag}), the composition $\Psi_{x, k} \circ a$ is surjective.  Thus, by commutativity, $b \circ \Psi_{x, k+1}$ is surjective.  Therefore, $c$ must be the zero map, so that the bottom sequence of maps between global sections is actually short exact.  Thus, by the five lemma, the center vertical map $\Psi_{x, k+1}$ is surjective. Thus, only the left vertical map $\alpha_{x, k}$ needs to be surjective in order to guarantee the normality of $\sct,$ as desired. 
\end{proof}

For the remainder of the section, we will focus on finding the conditions under which $\alpha_{x,k}$ is surjective.  The next two lemmas will help us better understand the target space.  

Since $F_x$ is a section over its image $\pi(F_x)$, we have the following short exact sequence:
\begin{equation}\label{seq} 0 \to \T \to \Ud \to N_{F_x /\H} \to 0. \end{equation}
We first calculate the left term of this sequence.
\begin{lem}\label{2isos}
Suppose $\L$ is 3-very ample. Let $b_x$ be the blow-up map of $X$ at $x$, and $E_x$ the corresponding exceptional divisor. Then 
\item $$\T \cong \det\E^* \big|_{F_x} \cong b_x^* \L^*(E_x).$$
\end{lem}

\begin{proof}

Consider the relative Euler sequence
\begin{equation} 0 \to \str_{\PE} \to \pi^* \E^* \otimes \str_{\PE}(1) \to T_{\PE / \H} \to 0. \end{equation}
Since $T_{\PE / \H}$ is a line bundle, taking determinants yields 
$$T_{\PE / \H} \cong \det (\pi^* \E^*) \otimes \str_{\PE}(2) \cong (\pi^* \det \E )^* \otimes \str_{\PE}(2).$$
So
$$\T \cong (\pi^* \det \E )^*\big|_{F_x} \otimes \str_{\PE}(2)\big|_{F_x} \cong \det \E^*\big|_{F_x} \otimes \str_{\PE}(2)\big|_{F_x}.$$

To calculate $\str_{\PE}(2)\big|_{F_x}$, consider the left diagram in (\ref{impdiags}). First note that by construction of the map $f: \PE \to \P^r$ via maps of vector bundles, it follows that the pullback of the tautological bundle is also the tautological bundle.  That is, $$f^*\str_{\P^r}(1) \cong \str_{\PE} (1).$$ Thus, 
$\str_{\PE}(1)\big|_{F_x}$ is isomorphic to the pullback of $\str_{\P^r}(1)\big|_x \cong \str_x$ to $F_x$. So $$\str_{\PE} (1) \big|_{F_x} \cong \str_{F_x}.$$ Thus, $$\T \cong \det\E^* \big|_{F_x},$$ which is the first isomorphism in the lemma.  The next step is to understand the restriction of $\E$ to $F_x$.

Consider the diagram
\begin{equation}
\xymatrix@=22pt {\sigma \iv (F_x) \ar@{^{(}->}[r]^i \ar[d]_{\sigma} & \Phi \ar[d]^{\sigma} \\ F_x \ar@{^{(}->}[r]_j & \H}.
\end{equation}
We have temporarily named the inclusion maps so that we can easily refer to them.  Note that $\sigma \iv (F_x)$ is two copies of $F_x$ intersecting along $E_x$. Since the above is a Cartesian square and $\sigma$ is flat and finite, base change yields
$$\E\big|_{F_x} = j^* \sigma_* q^* \L \cong \sigma_* i^* q^* \L.$$  If we think of $\Phi$ as $bl_\Delta (X \times X)$, then $q$ is the blowup morphism followed by projection to the first factor.  Thus, $i^* q^* \L$ is  isomorphic to  $\str_{F_x}$ when restricted to one reducible component, and $b_x^* \L$ when restricted to the other.  Thus, pushing forward, we have a natural map
\begin{equation}\label{natinj} \E\big|_{F_x} \cong \sigma_* i^* q^* \L \to \str_{F_x} \oplus b_x^* \L,\end{equation}
which is an injection that drops rank along $E_x$.

As an aside, it is useful to recall that the fiber of $\E\big|_{F_x}$ over a point $\xi \in F_x$ is $H^0(X, L \otimes \str_{\xi})$, where $\xi$ is some length two subscheme of $X$ which contains $x$.  So over generic $\xi$, the map (\ref{natinj}) on fibers is the sum of restrictions $$H^0(X, L \otimes \str_{\xi}) \to H^0(X, L \otimes \str_x) \oplus H^0(X, L \otimes \str_y),$$ where $\{x,y\} = \Supp (\xi)$.  

Since the vector bundles in (\ref{natinj}) have the same rank, taking determinants yields $$\det \E \big|_{F_x} \cong b_x^* \L(-E_x).$$ Thus, $$\T \cong \det\E^* \big|_{F_x} \cong b_x^* \L^*(E_x).$$
\end{proof}

We can now rewrite (\ref{seq}) as $$0 \to b_x^* \L^*(E_x) \to \Ud \to N_{F_x /\H} \to 0.$$  In this next key lemma, we calculate the middle term of this sequence.

\begin{lem}\label{splitlem}
Suppose $\L$ is 3-very ample.  Then for all $x \in X$, $$\U \cong \str_{F_x}^{\oplus n} \oplus (b_x^*\L (-2E_x)),$$ where $b_x$ is the blow-up map of $X$ at $x$, and $E_x$ is the corresponding exceptional divisor.
\end{lem}

\begin{proof}
The map induced by $\sigma$ on normal bundles $N_{F_x/ \Phi} \to N_{F_x / \H} $ is an isomorphism away from the ramification locus, which intersects $F_x$ in $E_x$.  Thus, $$\det N_{F_x/\H} \cong (\det N_{F_x/\Phi})(E_x).$$
Now $q : \Phi \to X$ is a smooth map of which $F_x$ is a fiber, so $N_{F_x/\Phi}$ is isomorphic to the pullback of $N_{x/X}$.  Thus, 
$$N_{F_x/\Phi} \cong  \str_{F_x}^n,$$ which means that $$\det N_{F_x/\H} \cong \str_{F_x}(E_x).$$  Looking back at the short exact sequence before the lemma, this means that $$\det \Ud \cong b_x^* \L^*(2E_x).$$

Now consider the following short exact sequence on normal bundles, again involving $\Ud$:
\begin{equation}\label{splits} 
0 \to N_{F_x/\Phi} \to \Ud \to N_{\Phi / \PE }\big|_{F_x} \to 0.
\end{equation}
We have already established that the left term is the trivial bundle of rank $n$.  Since $\Phi \subset \PE$ has codimension one, $N_{\Phi / \PE }\big|_{F_x}$ must be a line bundle.  Thus, taking determinants, we obtain
$$N_{\Phi / \PE }\big|_{F_x} \cong \det \Ud \cong b_x^* \L^*(2E_x).$$
We take the dual and rewrite (\ref{splits}) as 
$$
0 \to b_x^* \L(-2E_x)  \to \U \to \str_{F_x}^n \to 0.
$$

Our final goal is to show that the above sequence splits.  Since the right term is trivial, this is the same as showing that the map on global sections $$H^0\left(\U \right) \to H^0\left(N^*_{F_x/\Phi}\right)$$ is a surjection.
Consider the commutative diagram
\begin{equation*}
\xymatrix@=20pt { T^*_x \P^r  \arsurj[r] \ar[d]_{\alpha_{x,1}} & T^*_x X \ar[d] \\ H^0\left(\U \right) \ar[r] & H^0\left(N^*_{F_x/\Phi}\right)}.
\end{equation*}
As mentioned earlier, $$N_{F_x/\Phi} \cong T^*_x X \otimes \str_{F_x}.$$
Thus, the right vertical map is an isomorphism, so the bottom horizontal map must be a surjection, as desired.
Therefore, the desired sequence splits, which completes the proof.

\end{proof}

Now we return to showing that $\alpha_{x,k}$ is surjective.  In the case $k=1$, it is actually an isomorphism, which follows from a straight-forward geometric argument.

\begin{lem}\label{alphaone} Suppose $\L$ is 3-very ample.  Then $$\alpha_{x,1}: T^*_x \P^r \to H^0\left(\U \right)$$ is an isomorphism for all $x \in X$. \end{lem}

\begin{proof}
First we show  $\alpha_{x,1}$ is injective.  Let $w \in T^*_x \P^r$ be a nonzero covector.  Call the kernel hyperplane in the tangent space $H \subset \P^r$.  Since $X \in \P^r$ is non-degenerate, we can pick some $y \in X$ such that $y  \notin H$.  Define $\ell$ to be the secant line through $x$ and $y$.
Now define $$\tilde{\ell} := f\iv (\ell) \subset \PE.$$  Note that $\tilde{\ell}$ consists of all the points in $\PE$ in the fiber over the subscheme $x+y \in \H$.  That is, $$\tilde{\ell} = \pi\iv (x+y).$$ Thus, $\tilde{\ell}$ intersects $F_x \cong bl_x(X)$ at the point corresponding to $y$, i.e. at the point $(x+y, H^0(\L \otimes \str_{x+y}) \to H^0(\L \otimes \str_x)) \in \PE$.  Call this point $P_y$.  

Consider the commutative diagram of tangent spaces
\begin{equation*}
\xymatrix@=17pt {T_{P_y} \tilde{\ell} \ar[r]^{\cong} \arinj[d] & T_x \ell \arinj[d] \\ T_{P_y} \PE \ar[r]^{df} & T_x \P^r},
\end{equation*}
where the top horizontal map is an isomorphism since $f$ is an isomorphism on $\tilde{\ell}$.  Let $v \in T_{P_y} \tilde{\ell}$ be a nonzero vector.  Looking at the above diagram, $df(v)$ is nonzero and sits inside $T_x \ell$.  Thus, since $\ell$ is not contained in $H$, we know that $$\langle f^*w, v \rangle_{P_y} = \langle w, df(v) \rangle_x \neq 0,$$  which means that $f^*w \neq 0$.

Notice that the pullback map $T^*_x \P^r \to T_{P_y}^* \PE $ factors through $H^0(\U)$ as follows:
\begin{equation*}
\xymatrix@=20pt{T_x^*\P^r \ar[r]^{f^*} \ar[d]_{\alpha_{x,1}} & T^*_{P_y} \PE \\ H^0\left( \U\right) \ar[r]^{\textrm{restr.}} & H^0\left( \U \big|_{P_y}\right) \arinj[u]}
\end{equation*}
Thus, since $f^*w \neq 0$, we know $\alpha_{x,1}(w) \neq 0$.  Thus, $\alpha_{x,1}$ is injective.

Now to show that $\alpha_{x,1}$ is an isomorphism, we show that $T^*_x \P^r$ and $H^0\left(\U \right)$ have the same dimension.  

First of all, $$\dim T^*_x \P^r = r = h^0(\L) -1. $$  

Next, by Lemma \ref{splitlem}, $$h^0\left(\U \right) = h^0(\str_{F_x}^n) + h^0 (b_x^*\L(-2E_x)).$$  Of course, $h^0(\str_{F_x}^n) = n$.  To calculate $h^0 (b_x^*\L(-2E_x))$, consider the natural short exact sequence
$$0 \to \str_{F_x}(-2E_x) \to \str_{F_x} \to \str_{2E_x} \to 0.$$  Tensoring by $b_x^*\L$ and taking cohomology  yields $$0 \to H^0(b_x^*\L(-2E_x)) \to H^0(b_x^*\L) \to H^0(b_x^*\L\otimes \str_{2E_x}) \to \cdots.$$  Pushing forward, the second map on global sections is equal to the map $$H^0(\L) \to H^0(\L \otimes \str/m_x^2),$$ which is surjective by very ampleness of $\L$.  Thus, $$h^0 (b_x^*\L(-2E_x)) = h^0(\L) - h^0(\L \otimes \str/m_x^2) = h^0(\L) - (n+1).$$  So $$ h^0\left(\U \right) = n + h^0(\L) - n -1 = h^0(\L)-1 = \dim T^*_x \P^r,$$ as desired, which completes the proof.

\end{proof}

Now we prove the main theorem by showing that the higher $\alpha_{x,k}$ are surjective.

\begin{proof}[Proof of Theorem D]
By Lemma \ref{alpha}, showing that $$\alpha_{x,k} : \Sym^k(T_x^*\P^r) \to H^0(\Sym^k \U)$$ is surjective will prove the theorem.

Notice that we can build $\alpha_{x,k}$ from $\alpha_{x,1}$ as follows:
$$\xymatrix@=50pt {\Sym^k(T_x^*\P^r) \ar[r]^{\Sym^k \alpha_{x,1}} \ar[dr]_{\alpha_{x,k}} & \Sym^k H^0(\U) \ar[d]\\
& H^0(\Sym^k \U)
},$$ where the vertical map is the natural one.  By Lemma \ref{alphaone}, $\alpha_{x,1}$ is an isomorphism, so the induced map $\Sym^k \alpha_{x,1}$ must be as well.  Thus, $\alpha_{x,k}$ is surjective if and only if $$\Sym^k(H^0(\U)) \to H^0(\Sym^k \U)$$ is surjective.

By Lemma \ref{splitlem}, $$\Sym^k\left(H^0(\U)\right) \cong \Sym^k \left( H^0(\str_{F_x})^{\oplus n} \oplus H^0(b_x^*\L (-2E_x)) \right)$$ and $$H^0\left(\Sym^k \U\right) \cong H^0\left( \Sym^k \left( \str_{F_x}^{\oplus n} \oplus (b_x^*\L (-2E_x))\right) \right).$$

By construction of the map, $$\Sym^k(H^0(\U)) \to H^0(\Sym^k \U)$$ decomposes as the sum of maps of the form $$\Sym^i H^0 (b_x^*\L (-2E_x)) \to H^0\left((b_x^*\L (-2E_x))^{\otimes i}\right).$$  These maps are surjective for all $i$ if and only if $b_x^*\L (-2E_x)$ is normally generated, which is equivalent to the hypothesis of the theorem.  Thus, $$\Sym^k(H^0(\U)) \to H^0(\Sym^k \U)$$ is surjective, and we are done.
\end{proof}

\section{Corollaries}
In this section, we show how Theorem D can be applied to yield several results of geometric interest.

\begin{proof}[Proof of Corollary A]
Let $D$ be an effective divisor on $X$ of degree 4.  Then both $\L$ and $\L(-D)$ have degree greater than $2g-2$, so they are both non-special.  Thus, Riemann-Roch implies that $$h^0(\L(-D)) = h^0(\L)-4.$$  Thus, $\L$ is 3-very ample.

Let $x \in X$.  Then $$\deg\L(-2x) \geq 2g+1.$$  A classical result of Castelnuovo, Mattuck \cite{M}, and Mumford \cite{M2} states that line bundles on curves with degree at least $2g+1$ are normally generated, which means the maps in the hypothesis of the theorem are surjective, as desired.
\end{proof}

Next, we prove the corollary involving canonical curves.  Note that this example is not covered by Corollary A.

\begin{proof}[Proof of Corollary B]
Let $c = \textrm{Cliff}(X).$  The following classification is given in \cite{ELMS}:

\noindent $c=0 \iff X$ is hyperelliptic.

\noindent $c=1 \iff X$ has a $g^1_3$ or $X$ is a plane quintic.

\noindent $c =2 \iff X$ has a $g^1_4$ or $X$ is a plane sextic.

\noindent Thus, $c \geq 3$ if and only if $X$ has no $g_4^1$ and is not a plane sextic.

First we will show that $\omega_X$ is 3-very ample.  Let $D$ be an effective divisor of degree 4.  Then Riemann-Roch gives $$h^0(\omega_X (-D)) = h^0(D) + (2g-2-4) - g +1 = h^0(\omega_X) + h^0(D) -5.$$  Thus, $\omega_X$ is 3-very ample if and only if $h^0(D)=1$, i.e. $X$ has no $g^1_4$, which follows from the hypothesis.

Next we show that $\omega_X (-2x)$ is normally generated.  A theorem of Green and Lazarfeld (Theorem 1 in \cite{GL}) states that if $\L$ is very ample, and $$\deg \L \geq 2g+1 - 2 h^1(\L) - c,$$ then $\L$ is normally generated.  In the situation of interest, $\deg \omega_X (-2x) = 2g-4$, and by Serre duality $h^1(\omega_X (-2x)) = h^0(2x)$, which is 1 since $X$ is not hyperelliptic.  Thus, the Green-Lazarsfeld theorem implies $\omega_X(-2x)$ is normally generated as long as $c \geq 3.$
\end{proof}

In the above proof, the lack of a $g^1_4$ was equivalent to 3-very ampleness.  However, $c \geq 3$ merely implies the normal generation condition.  This raises the question: do we need the hypothesis that $X$ is not a plane sextic, or does the lack of a $g^1_4$ suffice?  In fact, if $X$ is a plane sextic, $\omega(-2x)$ is not normally generated.  This follows from a proof analogous to the one for Lemma 2.2 of \cite{K}, setting $D=\omega_X(-2x)$ and $k=2$.  We won't restate the proof, as it is nearly identical to Konno's proof except we replace $\ell$ with a line tangent to $X$ at $x$ and blow up twice at the intersection of $X$ and $\ell$ rather than once.  Thus, to satisfy the hypotheses of our theorem, it is necessary that $X$ is not a plane sextic.  However, our theorem only gives sufficient conditions for normality, so we ask the following question:

\begin{quest}
If $X$ is a smooth plane sextic, is $\Sigma(X,\omega_X)$ a normal variety?
\end{quest}

Now we turn to our final corollary, which deals with higher dimensional $X$.

\begin{proof}[Proof of Corollary C]
When $n=1$, $\L$ already has sufficiently high degree so that it satisfies the hypothesis of Corollary A.  We will assume from now on that $n$ is at least 2.

Example 1.8.23 of \cite{L2} states that $\omega_X \otimes  \A^{\otimes k} \otimes \B$ is very ample when $k \geq n+2$.  The main theorem of \cite{HTT} says that the product of an $i$-very ample line bundle with a $j$-very ample line bundle will be $(i+j)$-very ample.  Thus $\omega_X \otimes  \A^{\otimes k} \otimes \B$ will be 3-very ample for $k \geq n+4$.  For $n \geq 2$, we have $2(n+1) \geq n+4$, so $\L =\omega_X \otimes  \A^{\otimes 2(n+1)} \otimes \B$ must be 3-very ample.

Now we check the remaining hypotheses on $\widetilde{X} = bl_x X$.  First we calculate $b_x^*\L(-2E_x)$.
$$b_x^*\L = b_x^*\omega_X \otimes  b_x^*\A^{\otimes 2(n+1)} \otimes b_x^*\B = \omega_{\widetilde{X}}\otimes \str_{\widetilde{X}}(-(n-1)E) \otimes  b_x^*\A^{\otimes 2(n+1)} \otimes b_x^*\B.$$ Thus, we get 
$$b_x^*\L(-2E_x) = \omega_{\widetilde{X}} \otimes  b_x^*\A^{\otimes 2(n+1)}\otimes \str_{\widetilde{X}}(-(n+1)E_x) \otimes b_x^*\B 
= \omega_{\widetilde{X}} \otimes  (b_x^*\A^{\otimes 2}(-E_x))^{\otimes (n+1)} \otimes b_x^*\B .$$
$\A$ is very ample, so it is the restriction of $\str(1)$ of the corresponding projective space $\P^m$.  Consider the blowup $\widetilde{\P^m}$ of $\P^m$ at $x\in X$. It is well-known that $2\widetilde{H}-E$ is very ample, where $\widetilde{H}$ is the pullback of a hyperplane. Thus, $$\str_{\widetilde{X}}(2\widetilde{H}-E) = b_x^*\A^{\otimes 2}(-E_x)$$ is also very ample.  Furthermore, the pullback of a nef line bundle is again nef.  A theorem of Ein and Lazarsfeld in \cite{EL} states that line bundles of the form $\omega \otimes \mathcal{M}^{\otimes (n+1)} \otimes \mathcal{N}$, where $\mathcal{M}$ is very ample and $\mathcal{N}$ is nef, are normally generated.  Thus, $b_x^*\L(-2E_x)$ is normally generated, so $\sct$ must be normal.
\end{proof}

 \bibliographystyle{plain}
\bibliography{bibliography.bib}

\end{document}